\documentclass{amsart}

\usepackage{amssymb,amsmath}
\usepackage[colorlinks,plainpages,backref]{hyperref}
\usepackage{mathrsfs}
\usepackage[neveradjust]{paralist}
\usepackage[all]{xypic}
\xyoption{dvips}
\usepackage{epic,eepic}
\usepackage{url}

\theoremstyle{definition}
\newtheorem{ntn}{Notation}[section]
\newtheorem{dfn}[ntn]{Definition}
\theoremstyle{plain}
\newtheorem{lem}[ntn]{Lemma}
\newtheorem{prp}[ntn]{Proposition}
\newtheorem{thm}[ntn]{Theorem}

\theoremstyle{remark}

\newtheorem{rmk}[ntn]{Remark}
\newtheorem{exa}[ntn]{Example}

\newcommand{\boldone}{{\mathbf{1}}}
\newcommand{\bolda}{{\mathbf{a}}}

\newcommand{\boldu}{{\mathbf{u}}}

\newcommand{\del}{\partial}

\newcommand{\ideal}[1]{{\langle#1\rangle}}
\newcommand{\into}{\hookrightarrow}
\newcommand{\onto}{\twoheadrightarrow}

\newcommand{\calD}{\mathscr{D}}
\newcommand{\calE}{\mathscr{E}}
\newcommand{\calF}{\mathscr{F}}

\newcommand{\calM}{\mathscr{M}}
\newcommand{\calO}{\mathscr{O}}

\newcommand{\CC}{\mathbb{C}}
\newcommand{\NN}{\mathbb{N}}
\newcommand{\QQ}{\mathbb{Q}}

\newcommand{\ZZ}{\mathbb{Z}}

\renewcommand{\bar}{\overline}

\DeclareMathOperator{\Ext}{Ext}

\DeclareMathOperator{\Hom}{Hom}

\DeclareMathOperator{\mods}{-mods}

\DeclareMathOperator{\qdeg}{qdeg}
\DeclareMathOperator{\rk}{rk}
\DeclareMathOperator{\Res}{Res}
\DeclareMathOperator{\Sol}{Sol}

\DeclareMathOperator{\Spec}{Spec}
\DeclareMathOperator{\tdeg}{tdeg}
\DeclareMathOperator{\Var}{Var}
\DeclareMathOperator{\vol}{vol}

\begin{document}

\title[Resonance = reducibility]
{Resonance equals reducibility for $A$-hypergeometric systems}



\author{Mathias Schulze}
\address{
M. Schulze\\
Oklahoma State University\\
Department of Mathematics\\
Stillwater, OK 74078\\
USA}
\email{mschulze@math.okstate.edu}
\thanks{}

\author{Uli Walther}
\address{
U. Walther\\
Purdue University\\
Department of Mathematics\\
150 N.\ University St.\\
West Lafayette, IN 47907\\
USA}
\email{walther@math.purdue.edu}
\thanks{UW was supported by the NSF under grant DMS~0901123.}
\dedicatory{}


\begin{abstract}
Classical theorems of Gel'fand et al. and recent results of Beukers show that nonconfluent Cohen--Macaulay $A$-hypergeometric systems have reducible monodromy representation if and only if the continuous parameter is $A$-resonant.

We remove both the confluence and Cohen--Macaulayness conditions while simplifying the proof.
\end{abstract}

\subjclass{13N10,14M25,32S40}

\keywords{toric, hypergeometric, Euler--Koszul, D-module, resonance, monodromy}

\maketitle
\setcounter{tocdepth}{1}
\tableofcontents
\numberwithin{equation}{section}

\section{Introduction}\label{18}

In a series of seminal papers of the 1980's, Gel'fand, Graev, Kapranov and Zelevinski{\u\i} introduced {\em $A$-hypergeometric systems} $H_A(\beta)$, a class of maximally over\-determined systems of linear PDEs. 
These systems, today also known as \emph{GKZ-systems}, are induced by an integer $d\times n$-matrix $A$ and a parameter vector $\beta\in\CC^d$.

$A$-hypergeometric structures are nearly ubiquitous, generalizing most classical differential equations. 
Indeed, toric residues, generating functions for intersection numbers
on moduli spaces, and special functions (Gau\ss, Bessel, Airy, etc.)
may all be viewed as solutions to GKZ-systems, and the same is true for  varying Hodge structures on families of Calabi--Yau toric hypersurfaces as well as the space of roots of univariate polynomials with undetermined coefficients.

We shall identify $A$ with its set of columns $\bolda_1,\dots,\bolda_n$.
A parameter $\beta$ is \emph{nonresonant} if it is not contained in the locally finite subspace arrangement of \emph{resonant} parameters
\begin{equation}\label{17}
\Res(A):=\bigcup_{\tau}\left(\ZZ A+\CC\tau\right),
\end{equation}
the union being taken over all linear subspaces $\tau\subseteq \QQ^n$ that form a boundary component of the rational polyhedral cone $\QQ_+ A$.

Assuming that the toric ring $\CC[\NN A]=\CC[\bolda_1,\dots,\bolda_n]$ is Cohen--Macaulay and standard graded (the latter is equivalent to the classical notion of nonconfluence, see \cite{SW08}),
Gel'fand et al.~\cite{GKZ89,GKZ90} proved the following fundamental theorems:

\begin{enumerate}[(I)]
\item\label{1a} $H_A(\beta)$ is holonomic; 
\item\label{1b} the rank (dimension of the solution space) of $H_A(\beta)$ equals the degree of $\CC[\NN A]$ for generic $\beta$; 
\item\label{1c} if $\beta$ is nonresonant, the monodromy representation of the solutions of $H_A(\beta)$ in a generic point is irreducible.
\end{enumerate}

More recent research has shown that statements \eqref{1a} and \eqref{1b} hold true  irrespective of whether $\CC[\NN A]$ is Cohen--Macaulay or standard graded, \cite{Ado94,SST00,MMW05}. 
In Theorems~\ref{8} and \ref{9}, we prove the same of statement \eqref{1c} while providing a converse inspired by \cite{Beu10}.

The crucial tool for the proof of \eqref{1c} in \cite[Thm.~2.11]{GKZ90} is the Riemann--Hilbert correspondence of Kashiwara and Mebkhout, relating regular holonomic $D$-modules to perverse sheaves.
Confluence (i.e., irregularity) of $M_A(\beta)$
rules out the use of the Riemann--Hilbert correspondence
in the general case.

A powerful way of studying $H_A(\beta)$ is to consider the corresponding $D$-module $M_A(\beta)$ on $\CC^n$ as a $0$-th homology of the \emph{Euler--Koszul complex} $K_\bullet(\CC[\NN A],\beta)$.
This idea can be traced back to \cite{GKZ89} and was developed into a functor in \cite{MMW05}. 
Results from \cite{MMW05} show that $K_\bullet(\CC[\NN A],\beta)$ is a
resolution of $M_A(\beta)$ if and only if $\beta$ is not in the \emph{$A$-exceptional arrangement} $\calE_A$ (see Remark~\ref{30}), a well-understood (finite) subspace arrangement of $\CC^n$ comprised of the parameters $\beta$ for which the solution space of $H_A(\beta)$ is unusually large.

Surprisingly, the Euler--Koszul technique combined with the $D$-module/represen\-tation-theoretic description of GKZ-systems from \cite{SW09} serves as a replacement for the Riemann--Hilbert correspondence in the proof of \eqref{1c}.
This provides an approach that is simultaneously conceptually simpler
and more widely applicable. 

\subsubsection*{Acknowledgments}
We are grateful to the referees for their comments, and for informing us that Mutsumi Saito has an article in press with \emph{Compositio~Math.}\ that also discusses reducibility of GKZ-systems (in much greater detail). 
We would also like to thank Alan Adolphson for raising a relevant question.

\section{Hypergeometric system and Euler--Koszul homology}\label{15}

\subsection{Hypergeometric D-module}

Let $A=(a_{i,j})\colon\ZZ^n\to\ZZ^d$ be an integer $d\times n$-matrix, which we view both as a map, and as the finite subset $\{\bolda_1,\dots,\bolda_n\}$ of columns. 
We assume that the additive group $\ZZ A$ generated by the columns of $A$ is the free Abelian group $\ZZ^d$, but we do not assume that $A$ is positive, i.e., we do allow nontrivial units in the semigroup $\NN A$ (see Remarks~\ref{41} and \ref{42}).

Let $x_A=x_1,\dots,x_n$ be coordinates on $X:=\CC^n$, and let $\del_A=\del_1,\dots,\del_n$ be the corresponding partial derivative operators on $\CC[x_A]$. 
Then the \emph{Weyl algebra}
\[
D_A=\CC\ideal{x_A,\del_A\mid[x_i,\del_j]=\delta_{i,j},\,[x_i,x_j]=0=[\del_i,\del_j]}
\]
is the ring of algebraic differential operators on $\CC^n$.
With $\boldu_+=(\max(0,u_j))_j$ and $\boldu_-=\boldu_+-\boldu$, write $\square_\boldu$ for $\del^{\boldu_+}-\del^{\boldu_-}$, where here and elsewhere we freely use multi-index notation. 
The \emph{toric relations of $A$} are then
\[
\square_A:=\{\square_\boldu \,\mid\, A\boldu=0\}\subseteq R_A:=\CC[\del_A],
\]
and generate the \emph{toric ideal} $I_A=R_A\cdot\square_A$, whose residue ring is the \emph{toric ring} 
\[
S_A:=R_A/I_A\cong\CC[\NN A]=\CC[\bolda_1,\dots,\bolda_n].
\]
The \emph{Euler vector fields} $E=E_1,\dots,E_d$ induced by $A$ are defined as
\[
E_i:=\sum_{j=1}^na_{i,j}x_i\del_j.
\]
Then, for $\beta\in\CC^d$, the \emph{$A$-hypergeometric ideal} and \emph{$D$-module} are by \cite{GGZ87,GKZ89} the left $D_A$-ideal and -module
\[
H_A(\beta)=D_A\cdot\{E-\beta\}+D_A\cdot\square_A\quad\text{and}\quad
M_A(\beta)=D_A/H_A(\beta).
\]
The structure of the solutions to $H_A(\beta)$ is tightly interwoven with the combinatorics of the pair $(A,\beta)\in(\ZZ A)^n\times\CC A$ (see, for instance, \cite{ST98,CAD99,MM06,Oku06,Ber08}).

\begin{rmk}\label{41}
Suppose we were to weaken the condition $\ZZ A=\ZZ^d$ to ``the rank of $\ZZ A$ is $d$\,''. 
Pick a basis $B$ for $\ZZ A$, interpreted as elements of $\ZZ^d$. 
In terms of $B$, $A$ takes the form of the $d\times n$ matrix $A'$ (say) which satisfies $A=BA'$ and $\ZZ A'=\ZZ^d$. 
Choose $\beta\in\CC A=\CC A'$. 
The hypergeometric systems $H_A(\beta)$ and $H_{A'}(B^{-1}\beta)$ are equivalent since $\ker_{\ZZ^n}(A)=\ker_{\ZZ^n}(A')$.
\end{rmk}

\subsection{Torus action}\label{33}

Consider the algebraic $d$-torus $T:=\Spec(\CC[\ZZ A])\cong(\CC^*)^d$ with coordinate functions $t=t_1,\dots,t_d$. 
The columns $\bolda_1,\dots,\bolda_n$ of $A$ can be viewed as characters $\bolda_i(t)=t^{\bolda_i}$ on $T$, and the parameter vector $\beta\in\CC^d$ as a character on its Lie algebra via $\beta(t_i\del_{t_i})=-\beta_i+1$. 
These characters define an action of $T$ on $X^*:=\Spec(\CC[\NN^n])$, interpreted as the cotangent space $T^*_0X$ of $X$ at $0$, by
\[
t\cdot\del_A=(t^{\bolda_1}\del_1,\dots,t^{\bolda_n}\del_n).
\]
The toric ideal $I_A$ is the ideal of the closure of the orbit $T\cdot\boldone_A$ of $\boldone_A=(1,\ldots,1)$ in $X^*$, whose coordinate ring is $S_A$.

The contragredient action of $T$ on the coordinate ring $R_A$ of
$X^*$ is  given by
\[
(t\cdot P)(\del_A)=P(t^{-\bolda_1}\del_1,\dots,t^{-\bolda_n}\del_n)
\]
for $P\in R_A$. It yields a \emph{$\ZZ A$-grading} on $R_A$ on the
coordinate ring $\CC[x_A,\del_A]$ of $T^*X$:
\begin{equation}\label{10}
-\deg(\del_j)=\bolda_j=\deg(x_j).
\end{equation}
In particular, $\deg(\del^\boldu)=A\boldu$, and $E-\beta$ and $\Box_A$
are homogeneous.

The following description of $M_A(\beta)$ was given in \cite{SW09}.
Consider the algebraic $\calD_T$-module
\[
\calM(\beta):=\calD_T/\calD_T\cdot\ideal{\del_tt+\beta},
\]
where $\del_tt:=\del_1t_1,\dots,\del_dt_d$.
It is $\calO_T$-isomorphic to $\calO_T$ but equipped with a twisted $\calD_T$-module structure expressed symbolically as
\[
\calM(\beta)=\calO_T\cdot t^{-\beta-1}
\]
on which $\calD_T$ acts via the product rule.
The orbit inclusion
\[
\phi\colon T\to T\cdot \boldone\into X
\]
gives rise to a (derived) direct image functor $\phi_+\colon\calD_T\mods\to \calD_X\mods$. 
On $X$ one has access to the \emph{Fourier transform}: $\calF(x_i)=\del_i$, $\calF(\del_i)=-x_i$.
By \cite[Prop.~2.1]{SW09}, $\calF\circ \phi_+\calM(\beta)$ is represented by the Euler--Koszul complex $K_\bullet(S_A[\del_A^{-1}],\beta)$.
Thus, the latter is quasiisomorphic to $K_\bullet(S_A,\beta)$ if $\beta\not\in\Res(A)$ by \cite[Thm.~3.6]{SW09} and hence \cite[Cor.~3.8]{SW09} yields
\begin{equation}\label{43}
M_A(\beta)=\calF\circ\phi_+\calM(\beta)\text{ if }\beta\not\in\Res(A).
\end{equation}

\subsection{Euler--Koszul functor}\label{40}

We say that $\beta\in\ZZ A$ is a \emph{true degree} of the graded $R_A$-module $M$ if $\beta$ is the degree of a nonzero homogeneous element of $M$. 
The \emph{quasidegrees} of $M$ are the points $\qdeg(M)$ in the Zariski closure of $\tdeg(M)\subseteq\ZZ A\subseteq\CC A$.  

A graded $R_A$-module $M$ is called a \emph{toric module} if it has a finite filtration by graded $R_A$-modules such that each filtration quotient is a finitely generated $S_A$-module.  
The toric modules with $\ZZ A$-homogeneous maps of degree zero form a category that is closed under subquotients and extensions.
For every toric module the quasidegrees form a finite subspace arrangement where each participating subspace is a shift of a complexified face of $\QQ_{\ge 0}A$ by a lattice element.

For all $\beta\in\CC^d$ and for any toric $R_A$-module $M$ one can define a collection of $d$ commuting $D_A$-linear endomorphisms denoted $E_i-\beta_i$, $1\le i\le d$, on the $D_A$-module $D_A\otimes_{R_A}M$ which operate on a homogeneous element $m\in D_A\otimes_{R_A} M$ by $m\mapsto
(E_i-\beta_i)\circ m$, where
\[
(E_i-\beta_i)\circ m=
(E_i-\beta_i-\deg_i(m))\cdot m.
\]
There is an exact functor $K_\bullet(-,\beta)=K_\bullet(-,E-\beta)$ from the category of graded $R_A$-modules to the category of complexes of graded $D_A$-modules; it sends $M$ to the Koszul complex defined by all morphisms $E_i-\beta_i$. 
On toric modules, the functor returns complexes with holonomic homology.  
A short exact sequence
\[
0\to M'\to M\to M''\to 0
\]
of graded $R_A$-modules with homogeneous maps of degree zero induces a long exact sequence of \emph{Euler--Koszul homology}
\[
\cdots\to H_i(M'',\beta)\to H_{i-1}(M',\beta)\to
H_{i-1}(M,\beta)\to H_{i-1}(M'',\beta)\to\cdots
\]
where $H_i(-,\beta)=H_i(K_\bullet(-,\beta))$. 
If $M=S_A$ then $H_0(M,\beta)=M_A(\beta)$.

We refer to \cite{MMW05,SW09} for more details.

\subsection{Rank (jumps) and monodromy reducibility}

We shall write $D_A(x_A)$ for the ring of $\CC$-linear differential
operators on $\CC(x_A)$; note that $D_A(x_A)
=\CC(x_A)\otimes_{\CC[x_A]}D_A$ as left $D_A$-module. We further set
$M(x_A):=\CC(x_A)\otimes_{\CC[x_A]}M$ for any $D_A$-module $M$.  

The \emph{rank} $\rk(M)$ of a $D_A$-module $M$ is the
$\CC(x_A)$-dimension of $M(x_A)$. 
By Kashiwara's Cauchy--Kovalevskaya Theorem (see \cite[Thm.~1.4.19]{SST00}), it equals the $\CC$-dimension of the \emph{solution space} $\Sol(M)=\Hom_{D_A}(M,\CC\{x_A-\varepsilon\})$ of $M$ with coefficients in the convergent power series near the generic point $x_A=\varepsilon$ in (the analytic space associated to) $X$.

\begin{rmk}\label{30}
By \cite[Thm.~5.15]{Ado94} and \cite[Thms.~2.9, 7.5]{MMW05},
\[
\rk M_A(\beta)\ge\vol_A(A)
\]
with equality for generic $\beta\in\CC^n$.
Here $\vol_A(G)$ denotes, for any $G\subseteq\ZZ A$, the simplicial
volume of the convex hull of $G$ taken in the lattice $\ZZ A$.
More precisely, equality is equivalent to $\beta\not\in\calE_A$ where
\[
\calE_A:=\sum_{j=1}^n\bolda_j -
\bigcup_{i=0}^{d-1}\qdeg\left(\Ext^{n-i}_{R_A}(S_A,R_A)\right)
\]
is the \emph{exceptional arrangement}.  
\end{rmk}

\begin{dfn}
We say that a $D_A$-module $M$ has \emph{irreducible monodromy} if $M(x_A)$ is an irreducible $D_A(x_A)$-module (i.e.\ it has no nontrivial $D_A(x_A)$-quotients).
\end{dfn}

By \cite[Thm.~3.15]{Wal07}, monodromy irreducibility of $M(\beta)$ is a property of the equivalence class $\beta\in\CC A/\ZZ A$.

The nomenclature is based on the Riemann--Hilbert correspondence:
$D_A(x_A)$-quotients of $M(x_A)$ correspond to monodromy-invariant subspaces of $\Sol(M)$ in nonsingular points of $M$. 
(Analytic continuations of an analytic germ satisfy the same differential equations as the germ itself).

\begin{rmk}\label{42}
Careful reading of~\cite{MMW05} reveals that all fundamental results obtained through Euler--Koszul technology do not require $\NN A$ to be a positive semigroup. 
As a matter of fact, $\calE_A$ was defined in \cite{MMW05} in terms of local cohomology with supports at the origin of $X^*$; the translation between this definition and ours here can only be done if $A$ is pointed. 
On the other hand, it is the Ext-based definition that is (implicitly) used in all proofs in loc.~cit. 

In consequence, the main theorems in \cite{Wal07} and \cite{SW09} remain true in the absence of positivity since the only ingredients in their proofs that are specific to the hypergeometric situation are those of \cite{MMW05}. 
\end{rmk}

\section{Pyramids and resonance centers}

\begin{dfn}
For any subset $F$ of the columns of $A$ we write $\bar F$ for the complement $A\smallsetminus F$.

A \emph{face of $A$} is any subset $F\subseteq A$ subject to the condition that there be a linear functional $\phi_F\colon \ZZ A\to\ZZ$
that vanishes on $F$ but is positive on $\bar F$. This includes $F=A$
as possibility. Every face contains all units of $\NN A$, and
 $A$ is positive if and only if the empty set is a face of $A$.

For a given face $F$, we set
\[
I^F_A:=I_A+R_A\cdot \del_{\bar F}
\]
and note that $R_A/I^F_A=S_F$ as $R_A$-module.
\end{dfn}

\begin{dfn}
Let $F$ be a face of $A$.
The parameter $\beta\in\CC^d$ is \emph{$F$-resonant} if $\beta\in \ZZ A+\CC G$ for a proper subface $G$ of $F$.

If $\beta$ is $G$-resonant for all faces $G$ properly containing $F$, but not for $F$ itself, we call $F$ a \emph{resonance center for $\beta$}.
\end{dfn}

A resonance center is a minimal face $F$ for which $\beta\in \ZZ A+\CC F$. 
Every parameter $\beta$ has a resonance center; $A$ is a (and then the only) center of resonance for $\beta$ if and only if $\beta$ is nonresonant in the
usual sense (i.e., $\beta\not\in\Res(A)$, defined in \eqref{17}). 
On the other hand, for positive $A$, the empty face is a resonance center for $\beta$ if and only if $\beta\in\ZZ A$.

\begin{exa}
It is easy to have several resonance centers for $\beta$. 
For example, consider $\beta=(\frac12,1)$ on the quadric cone $A=\begin{pmatrix}1&1&1\\0&1&2\end{pmatrix}$; $\beta$ has both extremal rays as resonance centers.
\end{exa}

\begin{dfn}
We say that $A$ is a(n iterated) \emph{pyramid over the face $F$} if  $d=\dim_\ZZ(\ZZ A)$ equals $|\bar F|+\dim_\ZZ(\ZZ F)$. 
\end{dfn}

The following equivalences are trivial or follow from \cite[Lem.~3.13]{Wal07}.

\begin{lem}\label{4}
The following are equivalent:
\begin{enumerate}
\item $F$ is a face and $A$ is a pyramid over $F$;
\item $\bolda_j\not\in\QQ (A\smallsetminus \{\bolda_j\})$ for any $j\not\in F$;
\item\label{5} $\ZZ A=\ZZ \bolda_j\oplus\ZZ(A\smallsetminus \{\bolda_j\})$ for any $j\not\in F$;
\item $\vol_F(F)=\vol_A(A)$;
\item\label{22} for every $\beta\in\CC A$, the coefficients $c_j$ in the sum $\beta=\sum_A c_j\bolda_j$ are uniquely determined by $\beta$ for $j\not\in F$;
\item the generators $\square_A$ of $I_A$ do not involve $\del_j$ for any $j\not\in F$;
\item $S_F\otimes_{\CC}\CC[\del_{\bar F}]=S_A$ as $R_A$-modules.
\end{enumerate}
\end{lem}

\begin{ntn}
Suppose $F$ is any nonempty face of $A$, and let $X_F$, $X^*_F$, $T_F$, $H_\bullet^F$, etc.\ be defined as in Section~\ref{15} with $A$ replaced by $F$ (cf.~Remark~\ref{41} for the case where $\ZZ A/\ZZ F$ has torsion).
Write $E^F=E^F_1,\dots,E^F_d$ where $E^F_i:=\sum_{j\in F}a_{i,j}x_j\del_j$ is the part of $E_i$ supported in $F$. 
Then, in particular, $M_F(\beta)=D_F/(D_F\cdot\ideal{E^F-\beta}+D_F\cdot I_F)$ for $\beta\in\CC F$.
\end{ntn}

Suppose now that $A$ is a pyramid over the face $F$, and let $\beta\in\CC A$. 
The splitting in Remark~\ref{4}.\eqref{5} corresponds to a splitting of tori $T_A=T_F\times\prod_{\bolda_j\in\bar F}T_{\bolda_j}$ which in turn gives a splitting of the spaces of Lie algebra characters $\CC A=\CC F\oplus\bigoplus_{\bolda_j\in\bar F}\CC\bolda_j$.
Then $\beta$ decomposes correspondingly as
\[
\beta=\beta^F+\sum_{j\in \bar F}\beta^{\bar F}_j.
\] 
Let $\iota_F\colon X^*_F\into X^*_A$ be the inclusion.
By \cite[Lem.~4.8]{MMW05}, for $\beta\in\CC F$, 
\begin{align}\label{29}
(\calF\circ\iota_{F,+}\circ\calF^{-1})M_F(\beta)
&=\CC[x_{\bar F}]\otimes_\CC M_F(\beta)\\
\nonumber&\cong H_0(S_F,\beta)=D_A/(D_A\cdot\ideal{E^F-\beta}+D_A\cdot I_A^F)
\end{align}
as $D_A$-modules. 
In the following lemma, \eqref{27} follows from \eqref{26} and
\eqref{29} above.

\begin{lem}\label{6}
If $A$ is a pyramid over $F$ then the following conditions hold:
\begin{enumerate}
\setcounter{enumi}{7}
\item\label{26} the ideal $H_A(\beta)$ contains $x_j\del_j-\beta^{\bar F}_j$ for $j\not\in F$;
\item\label{27} $M_A(\beta)(x_A)=\CC(x_A)\otimes_{\CC[x_F]}M_F(\beta)$ for $\beta\in\CC F$; 
\item\label{24} the solutions of $M_A(\beta)$ are the solutions of $M_F(\beta^F)$, multiplied with the unique solution to the system
\[
\{x_j\del_j\bullet f=\beta^{\bar F}_j\cdot f\}_{j\in\bar F}.
\]
\end{enumerate}
In particular, $\beta\in\calE_A$ if and only if $\beta^F\in\calE_F$.
\end{lem}

\begin{prp}
If $\beta\in\CC A$ has a resonance center $F$ over which $A$ is a
pyramid, then $F$ is the only resonance center for $\beta$. 
\end{prp}

\begin{proof}
Let $G$ be a second resonance center for $\beta$ and suppose $G$ meets the complement of $F$; pick $\bolda_k\in G\cap\bar F$. 
Since $\ZZ \bolda_k$ is a direct summand of $\ZZ A$, it is also a direct summand
of $\ZZ G$. It follows that $G\smallsetminus \{\bolda_k\}$ is a face $G'$ of
$A$. 

As $F$ and $G$ are resonance centers, 
\[
\beta=z_k\bolda_k+\sum_{j\in\bar F\smallsetminus\{k\}}z_j\bolda_j+\sum_{j\in F}c_j \bolda_j,\quad
\beta=c'_k\bolda_k+\sum_{j\in\bar  G'\smallsetminus\{k\}} z'_j\bolda_j+\sum_{j\in G'}c'_j\bolda_j
\]
where $z_k,z_j,z'_j\in \ZZ$ and $c'_k,c_j,c'_j\in\CC$. 
By Lemma~\ref{4}.\eqref{22}, the coefficients for $\bolda_k$ in these sums are identical, $c'_k=z_k\in\ZZ$. 
It follows that
\[
\beta=\left(z_k\bolda_k+\sum_{j\in\bar  G'\smallsetminus\{k\}}z'_j\bolda_j\right)+\sum_{G'}c'_j\bolda_j\in\ZZ A+\CC G'.
\]
This contradicts $G$ being a resonance center. Thus $G\cap\bar
F=\emptyset$ and so $G\subseteq F$. 
But then $F$ can only be a resonance center if $F=G$.
\end{proof}

\section{Resonance implies reducibility}

The following result generalizes Theorem~3.4 in \cite{Wal07} and Theorem~1.3 in \cite{Beu10}.

\begin{thm}\label{8}
Let $F$ be a resonance center for $\beta\in\CC A$.
If $A$ is not a pyramid over $F$ then $M_A(\beta)$ has reducible monodromy.
\end{thm}

\begin{proof}
By hypothesis, we have $\beta-\gamma\in\ZZ A$ for some $\gamma\in\CC F$.
We first dispose of the case $F=\emptyset$. 
In that case, $A$ is positive, $\gamma=0$, $\beta\in\ZZ A$ and, by \cite[Thm.~3.15]{Wal07}, we may assume $\beta=0$. 
Then $\CC(x_A)$ is a rank-$1$ quotient of $M_A(\beta)(x_A)$.
But $A$ is not a pyramid over $F$, so 
\[
\rk(M_A(\beta))\ge\vol_A(A)>\vol_F(F)=1=\rk(\CC(x_A))
\]
by Remark~\ref{30} and Lemma~\ref{4}.
So $\CC(x_A)$ is a proper quotient of $M_A(\beta)(x_A)$, and hence $M_A(\beta)$ has reducible monodromy.
We can hence assume that $F$ is not empty, and by \cite[Thm.~3.15]{Wal07}, we need to show the reducibility of $M_A(\gamma)$.

Consider the surjection 
\[
M_A(\gamma)=H_0(S_A,\gamma)\onto H_0(S_F,\gamma)
\]
induced by the surjection $S_A\onto S_F$. 
Therefore, it suffices to show that $0<\rk(H_0(S_F,\gamma))<\vol_A(A)$ by Remark~\ref{30}.
Since $F$ is a resonance center for $\beta$, and hence for $\gamma$ as well, $\gamma$ is a nonresonant parameter for the GKZ-system
\[
M_F(\gamma)=D_F/(D_F\cdot\ideal{E^F-\gamma}+D_A\cdot I_F).
\]
Then, by Remark~\ref{30}, $\rk(M_F(\gamma))=\vol_F(F)>0$ and $\rk(M_A(\gamma))\ge\vol_A(A)$.
As $A$ is not a pyramid over $F$, $\vol_F(F)<\vol_A(A)$ by Lemma~\ref{4}.
Finally, $\rk(M_F(\gamma))=\rk(H_0(S_F,\gamma))$  by \eqref{29}.
Combining the above (in)equalities yields the claim.
\end{proof}

\section{Resonance follows from reducibility}\label{13}

We now generalize Theorem~2.11 in \cite{GKZ90}.

\begin{thm}\label{9}
Let $F$ be a resonance center for $\beta$. 
If $A$ is a  pyramid over $F$ then $M_A(\beta)$ has irreducible monodromy.
\end{thm}

\begin{proof}

First consider the case $F=A$. 
Then $\beta\not\in\Res(A)$ and hence
$M_A(\beta)=\calF\circ\phi_+(\calM_\beta)$ by \eqref{43}.
As in the proof of \cite[Prop.~2.1]{SW09}, factor $\phi=\varpi\circ\iota$ into the closed embedding of tori
\begin{equation}\label{19}
\iota:T\into\Spec(\CC[\ZZ^n])=Y^*\cong(\CC^*)^n
\end{equation}
induced by $\ZZ A\subseteq\ZZ^n$, followed by the open embedding 
\begin{equation}\label{20}
\varpi:Y^*=X^*\smallsetminus\Var(\del_1\cdots\del_n)\into X^*.
\end{equation}
By Kashiwara equivalence, $\iota$ preserves irreducibility. 
The same holds for $\varpi$, because $D$-affinity of both the target and the source of the inclusion map allows to detect submodules on global sections. 
But global sections on $Y^*$ and $X^*$ agree because we are looking at an open embedding. 
Since $\calM(\beta)$ is clearly irreducible, $\phi_+\calM(\beta)$ is as well. 
As Fourier transforms preserve composition chains, $M_A(\beta)$ is irreducible.
It follows that $M_A(\beta)$ has irreducible monodromy.

Suppose now that $F$ is a proper face. 
Choose $\gamma\in\CC F$ with $\beta-\gamma\in\ZZ A$. 
Then $M_F(\gamma)$ is irreducible by the first part of the proof, and the claim follows from Lemma~\ref{6}.\eqref{27} and \cite[Thm.~3.15]{Wal07}.
Finally, if $F=\emptyset$ then $A$ is positive and Lemma~\ref{6}.\eqref{26} shows that $M_A(\beta)(x_A)=\CC(x_A)$ which has clearly irreducible monodromy.
\end{proof}


\bibliographystyle{amsalpha}
\bibliography{resgkz}

\end{document}